\documentclass{amsart}

\usepackage{amsmath}
\usepackage{amssymb}
\usepackage[]{amsrefs}
\usepackage{xpatch}
\usepackage{kantlipsum} 
\calclayout

\xpatchcmd{\proof}{\itshape}{\prooflabelfont}{}{}
\newcommand{\prooflabelfont}{\bfseries}

\DeclareMathOperator{\Gr}{gr}
\DeclareMathOperator{\Spec}{Spec}
\DeclareMathOperator{\E}{e}
\DeclareMathOperator{\Index}{index}
\DeclareMathOperator{\Tr}{tr}

\DeclareMathOperator{\Ulr}{ulr}
\DeclareMathOperator{\Eli}{eli}
\DeclareMathOperator{\Gll}{g\ell\ell}

\DeclareMathOperator{\Height}{height}

\DefineSimpleKey{bib}{primaryclass}{}
\DefineSimpleKey{bib}{archiveprefix}{}

\BibSpec{arXiv}{%
  +{}{\PrintAuthors}{author}
  +{,}{ \textit}{title}
  +{}{ \parenthesize}{date}
  +{,}{ arXiv }{eprint}
  +{,}{ primary class }{primaryclass}
}

\usepackage{kantlipsum} 
\calclayout
\usepackage{extarrows}
\usepackage{amssymb}
\usepackage[utf8]{inputenc}
\newtheorem{theorem}{Theorem}[section]
\newtheorem{proposition}[theorem]{Proposition}
\newtheorem{lemma}[theorem]{Lemma}
\newtheorem{corollary}[theorem]{Corollary}
\theoremstyle{definition}
\newtheorem{example}[theorem]{Example}
\theoremstyle{definition}
\newtheorem{definition}[theorem]{Definition}
\newtheorem{remark}[theorem]{Remark}
\newtheorem{question}[theorem]{Question}

\numberwithin{equation}{section}
\usepackage{amsthm,amsmath,amssymb}
\usepackage[all,cmtip]{xy}
\usepackage{amsmath}
\usepackage{amssymb}
\usepackage{amsthm}
\usepackage[]{amsrefs}
\usepackage{hyperref}
\usepackage{xpatch}
\usepackage{amsfonts}
\usepackage{amssymb}
\usepackage[utf8]{inputenc}
\usepackage{amsthm}
\usepackage{stmaryrd}
\usepackage{csquotes}
\usepackage{extarrows}
\MakeOuterQuote{"}
\theoremstyle{definition}
\begin{document}
\date{\today}

\title[The index of a numerical semigroup ring]{The index of a numerical semigroup ring}

\author[Richard F. Bartels]{Richard F. Bartels}

\address{Department of Mathematics, Trinity College, 300 Summit St, Hartford, CT, 06106}

\email{rbartels@trincoll.edu}

\urladdr{https://sites.google.com/view/richard-bartels-math/home}



\title[Some properties of ideals in Cohen-Macaulay local rings]{Some properties of ideals in Cohen-Macaulay local rings}
\begin{abstract}
For a Cohen-Macaulay local ring $(R,\mathfrak{m})$ with canonical module, we study how relations between $\Index(R)$ and $\Gll(R)$ and between $\Index(R)$ and $e(R)$ are preserved when factoring out regular sequences and localizing at prime ideals. We then give conditions for when ideals in a one-dimensional Cohen-Macaulay local ring are Elias and Burch, and use these conditions to study the relationship between Elias, Burch, and Ulrich ideals.  
\end{abstract}
\maketitle
\large{
 \section{Introduction}\text{}
 Let $(R,\mathfrak{m},k)$ be a Cohen-Macaulay local ring. For a finitely-generated $R$-module $M$, Auslander's $\delta$-invariant, denoted $\delta_{R}(M)$, is the smallest non-negative integer $n$ such that there exists a surjective $R$-map $X \oplus R^{n} \longrightarrow M$, where $X$ is a maximal Cohen-Macaulay $R$-module with no free direct summand. For each $n \geq 1$, we have
 \[
 0 \leq \delta_{R}(R/\mathfrak{m}^{n}) \leq \delta_{R}(R/\mathfrak{m}^{n+1}) \leq 1
 \] [15, Corollary 11.28].
 The smallest positive integer $n$ for which $\delta_{R}(R/\mathfrak{m}^{n})=1$ is the following numerical invariant defined by Auslander.
\[ \text{index}(R):=\text{inf}\{n \geq 1\,|\,\delta_{R}(R/\mathfrak{m}^{n})=1\}\]

Suppose $R$ has a canonical module $\omega$. The {\it{trace}} of $\omega$ in $R$, denoted $\Tr_{R}\omega$, is the ideal of $R$ generated by all $R$-homomorphic images of $\omega$ in $R$. Ding proved that if $R$ is a Cohen-Macaulay local ring with canonical module such that $\mathfrak{m} \subseteq \Tr_{R}\omega$, then $\text{index}(R)$ is finite and bounded above by the Hilbert-Samuel multiplcity of $R$, denoted $e(R)$, and also the generalized Loewy length of $R$, denoted $\Gll(R)$ \,[9, Propositions 2.3 and 2.4]. The generalized Loewy length of a local ring $(R,\mathfrak{m})$ is the smallest positive integer $n$ for which $\mathfrak{m}^{n}$ is contained in the ideal generated by a system of parameters of $R$.
\newline

In this paper, we study how the inequalities $\Index(R) \leq e(R)$ and $\Index(R) \leq \Gll(R)$ are preserved when factoring out regular sequences and localizing at prime ideals. We prove that if ${\bf{x}}=x_1,...,x_r\in \mathfrak{m}$ is a regular sequence and $\mathfrak{m} \subseteq \Tr_{R}\omega$, then $\Index(\overline{R}) \leq e(\overline{R})$ and $\Index(\overline{R}) \leq \Gll(\overline{R})$, where $\overline{R}=R/{\bf{x}}R$. We also prove that if $\mathfrak{p} \in \Spec(R)$ is a prime ideal such that $\mathfrak{p}R_{\mathfrak{p}} \subseteq (\Tr_{R}\omega)R_{\mathfrak{p}}$, then $\Index(R_{\mathfrak{p}}) \leq e(R_{\mathfrak{p}})$ and $\Index(R_{\mathfrak{p}}) \leq \Gll(R_{\mathfrak{p}})$. \newline 

For all prime ideals $\mathfrak{p}$ that are analytically unramified, we obtain the bound $\Index(R_{\mathfrak{p}}) \leq e(R)$. If $R$ is an abstract hypersurface, or $R$ is a hypersurface with infinite residue field, this implies $\Index(R_{\mathfrak{p}}) \leq \Index(R)$ for all analytically unramified prime ideals $\mathfrak{p} \in \Spec(R)$. Using these inequalities, we prove conditions for when the equalities $\Index(R)=e(R)$ and $\Index(R)=\Gll(R)$ are preserved when factoring out regular sequences.\newline

Let $(R,\mathfrak{m})$ be a one-dimensional Cohen-Macaulay local ring with $s=\Index(R)$ \,finite. In [1, Theorem 2.3], we proved that if $R$ contains a nonzerodivisor $x \in \mathfrak{m}^{t} \setminus \mathfrak{m}^{t+1}$ such that the induced map 
\[\overline{x}:\mathfrak{m}^{i-1}/\mathfrak{m}^{i} \longrightarrow \mathfrak{m}^{i+t-1}/\mathfrak{m}^{i+t}\]
is injective for $1 \leq i \leq s$, then \[\Gll(R) \leq \Index(R)+t-1.\]\text{}\newline
We generalize this result to symbolic powers of height one prime ideals in Cohen-Macaulay local rings of arbitrary positive dimension. Let $(R,\mathfrak{m})$ be a Cohen-Macaulay local ring and let $\mathfrak{p} \in \Spec(R)$ be a height one prime ideal such that $s=\Index(R_{\mathfrak{p}})$ is finite. In Proposition 2.22, we prove that if $x \in \mathfrak{p}^{(t)} \setminus \mathfrak{p}^{(t+1)}$ is a nonzerodivisor such that the induced map
\[\overline{x}:\mathfrak{p}^{(i-1)}/\mathfrak{p}^{(i)} \longrightarrow \mathfrak{p}^{(i+t-1)}/\mathfrak{p}^{(i+t)}\]
is injective for $1 \leq i \leq s$, then \[\Gll(R_{\mathfrak{p}}) \leq \Index(R_{\mathfrak{p}})+t-1.\]\text{}

In section 3, we study the relationship between several classes of ideals in local rings of positive depth. We first consider a class of $\mathfrak{m}$-primary ideals in one-dimensional Cohen-Macaulay local rings $(R,\mathfrak{m},k)$, called Elias ideals. Elias ideals are defined by Dao in [4] as follows. For an arbitrary $\mathfrak{m}$-primary ideal $I$ of $R$, we have $\text{type}_{R}(R/I) \leq \text{type}_{R}(I)$, where the type of an $R$-module $M$ is given by 
\[
\text{type}_{R}(M)=\text{dim}_{k}\text{Ext}_{R}^{\text{dim}(M)}(k,M).
\] If $\text{type}_{R}(R/I)=\text{type}_{R}(I)$, we say that $I$ is Elias [4, Definition 1.1]. Dao gives several characterizations of Elias ideals and their main properties in [4, Theorem 1.2 and Corollary 1.3]. Notably, an $\mathfrak{m}$-primary ideal $I$ in a one-dimensional Gorenstein local ring $R$ is Elias if and only if $\delta_{R}(R/I)=1$ [4, Proposition 4.1].\newline

Using techniques from De Stefani's proof of [7, Proposition 2.10], we prove in Proposition 3.11 that an $\mathfrak{m}$-primary ideal $I$ in a one-dimensional Cohen-Macaulay local ring $(R,\mathfrak{m})$ is Elias if $x \in \mathfrak{m}(xR:_{Q}I)$ for some nonzerodivisor $x \in \mathfrak{m}$. Here, $Q$ denotes the total ring of fractions of $R$. In [4], Dao proves that this condition characterizes Elias ideals for one-dimensional Gorenstein local rings [4, Theorem 1.2 (7)]. We give an example to show that this condition does not characterize Elias ideals for one-dimensional Cohen-Macaulay local rings.\newline

We give a natural corollary to Proposition 3.11 in Corollary 3.15, showing that if $I$ and $J$ are ideals of one-dimensional Cohen-Macaulay local ring $(R,\mathfrak{m})$ with $I$ an $\mathfrak{m}$-primary ideal, and if $IJ \subseteq xR$ and $x \in J\mathfrak{m}$ for some nonzerodivisor $x \in \mathfrak{m}$, then $I$ is Elias.\newline

We also prove conditions for when ideals in a local ring of positive depth are Burch. An ideal $I$ of local ring $(R,\mathfrak{m})$ is Burch if $\mathfrak{m}I \neq \mathfrak{m}(I:_R \mathfrak{m})$ [6, Definition 2.1]. Burch ideals and Burch modules have been studied extensively-see [5], [6], [8], and [12].\newline

After proving conditions for when Elias and Ulrich ideals are Burch, we prove in Proposition 3.22 a result that builds on [5, Proposition 3.7] and allows us to further study the relationship between Elias, Burch, and Ulrich ideals. We prove that if $(R,\mathfrak{m})$ is a local ring with positive depth such that $J\mathfrak{m} \not\subseteq I\mathfrak{m}$ and $IJ \subseteq xR$ for some nonzerodivisor $x \in \mathfrak{m}$, then $I$ is Burch.\newline
\section{\large{Invariants of Cohen-Macaulay local rings}}\text{}
Let $(R,\mathfrak{m},k)$ be a Cohen-Macaulay local ring. In this section, we study how invariants of $R$, and relations between these invariants, are preserved when factoring out regular sequences and localizing at prime ideals. First, assuming $R$ has a canonical module $\omega$, we use the trace of $\omega$ to describe how established inequalities for $\Index(R)$ and $e(R)$ and for $\Index(R)$ and $\Gll(R)$ are preserved for such quotients and localizations. 
\smallskip

We then give conditions for when equality of $\Index(R)$ and $e(R)$ implies the equalities $\Index(R)=\Index(\overline{R})=e(\overline{R})$, and for when equality of $\Index(R)$ and $\Gll(R)$ implies the equalities  $\Index(R)=\Index(\overline{R})=\Gll(\overline{R})$, where ${\bf{x}}=x_1,...,x_r \in \mathfrak{m}$ is a regular sequence and $\overline{R}=R/{\bf{x}}R$. Finally, we generalize [1, Theorem 2.3 and Corollary 2.5] to symbolic powers of height one prime ideals in Cohen-Macaulay local rings of arbitrary positive dimension. \newline

Throughout this section, $(R,\mathfrak{m},k)$ is a (Noetherian) local ring.\smallskip
\begin{definition}
Let $M$ be an $R$-module. The {\it{trace}} of $R$ in $M$, denoted $\Tr_{R}M$, is the ideal of $R$ generated by all $R$-homomorphic images of $M$ in $R$.
\end{definition}\smallskip
Herzog, Hibi, and Stamate showed that the trace of a module is well-behaved under base change.\smallskip
\begin{lemma}[13, Lemma 1.5]
Let $R$ and $S$ be rings. Let $\varphi:R \longrightarrow S$ be a ring map and $M$ an $R$-module. Then \[(\Tr_{R}M)S \subseteq \Tr_{S}(M \otimes_{R}S).\]
\end{lemma}\smallskip
Using this lemma, we prove corollaries to the following proposition for quotients and localizations of Cohen-Macaulay local rings. In what follows, $e(R)$ denotes the Hilbert-Samuel multiplicity of $R$.
\begin{proposition}[9, Propositions 2.4]
Let $(R,\mathfrak{m})$ be a Cohen-Macaulay local ring with canonical module $\omega$. If $\mathfrak{m}\subseteq \Tr_{R}\omega$, then 
\[
\Index(R) \leq \Gll(R)
\] and 
\[
\Index(R) \leq \E(R).
\]
\end{proposition}\smallskip
\begin{corollary}
Let $(R,\mathfrak{m})$ be a Cohen-Macaulay local ring with canonical module $\omega$. Let ${\bf{x}}=x_1,...,x_r \in \mathfrak{m}$ be a regular sequence and $\overline{R}=R/{\bf{x}}R$. If $\mathfrak{m} \subseteq \Tr_{R}\omega$, then \newline
\[
\Index(\overline{R}) \leq \Gll(\overline{R})
\] and 
\[
\Index(\overline{R}) \leq \E(\overline{R}).
\]
\end{corollary}\smallskip
\begin{proof}
Let $S=\overline{R}$ in Lemma 2.2. Let $\omega_{\overline{R}}=\omega \otimes_{R}\overline{R}$ and $\overline{\mathfrak{m}}=\mathfrak{m}\overline{R}$. We have\smallskip 
\[
\overline{\mathfrak{m}} \subseteq (\Tr_{R}\omega)\overline{R} \subseteq \Tr_{\overline{R}}\omega_{\overline{R}}.
\]\smallskip Since $\omega_{\overline{R}}$ is the canonical module for Cohen-Macaulay local ring $(\overline{R},\overline{\mathfrak{m}})$, it follows that \[\Index(\overline{R}) \leq \Gll(\overline{R})\] and \[\Index(\overline{R}) \leq \E(\overline{R})\] by Proposition 2.3. 
\end{proof}\smallskip
\begin{definition}[17, page 114]
Let $(R,\mathfrak{m})$ be a local ring. We say that $R$ is {\it{analytically unramified}} if its $\mathfrak{m}$-adic completion is a reduced ring. We say that an ideal $I$ of $R$ is analytically unramified if $R/I$ is analytically unramified.
\end{definition}
\begin{definition}[10, page 1]
Let $(R,\mathfrak{m})$ be a local ring. We say that $R$ is an {\it{abstract hypersurface}} if there is a regular local ring $(A,\mathfrak{n})$ and a nonzero element $x \in \mathfrak{n}^2$ such that the $\mathfrak{m}$-adic completion of $R$ can be written as $A/xA$. If $R$ itself can be written as $A/xA$, then we simply say that $R$ is a hypersurface.
\end{definition}
\text{}
If $R$ is an abstract hypersurface, or a hypersurface with infinite residue field, then we have $\Index(R_{\mathfrak{p}}) \leq \Index(R)$ for each analytically unramified prime ideal $\mathfrak{p}$ of $R$.\newline
\begin{corollary}
Let $(R,\mathfrak{m},k)$ be a Cohen-Macaulay local ring with canonical module $\omega$. Let $\mathfrak{p} \in \Spec(R)$. 
\newline\newline
If $\mathfrak{p}R_{\mathfrak{p}} \subseteq (\Tr_{R}\omega)R_{\mathfrak{p}}$, then
\[
\Index(R_{\mathfrak{p}}) \leq \Gll(R_{\mathfrak{p}})
\] and 
\[
\Index(R_{\mathfrak{p}}) \leq \E(R_{\mathfrak{p}}).
\] If also $\mathfrak{p}$ is analytically unramified, then
\[
\Index(R_{\mathfrak{p}}) \leq \E(R).
\]\text{}\newline If $R$ is a hypersurface, or $k$ is infinite and $R$ is an abstract hypersurface, then for each analytically unramified prime ideal $\mathfrak{p} \in \Spec(R)$, we have \smallskip
\[
\Index(R_{\mathfrak{p}}) \leq \Index(R).
\]
\end{corollary}\smallskip
\begin{proof}
Let $S=R_{\mathfrak{p}}$ in Lemma 2.2 and let $\omega_{\mathfrak{p}}= \omega \otimes_{R}R_{\mathfrak{p}}$. We have
\[
\mathfrak{p}R_{\mathfrak{p}} \subseteq (\Tr_{R}\omega)R_{\mathfrak{p}} \subseteq \Tr_{R_{\mathfrak{p}}}\omega_{\mathfrak{p}}.
\] Since $\omega_{\mathfrak{p}}$ is the canonical module for Cohen-Macaulay local ring $(R_{\mathfrak{p}}, \mathfrak{p}R_{\mathfrak{p}})$, it follows that $\Index(R_{\mathfrak{p}}) \leq \Gll(R_{\mathfrak{p}})$ and $\Index(R_{\mathfrak{p}}) \leq \E(R_{\mathfrak{p}})$ by Proposition 2.3. If $\mathfrak{p}$ is analytically unramified, then we have $\E(R_{\mathfrak{p}}) \leq \E(R)$ by [17, Theorem 40.1]. Therefore, $\Index(R_{\mathfrak{p}}) \leq \E(R)$. If $R$ is a hypersurface, or $k$ is infinite and $R$ is an abstract hypersurface, then $\Index(R)=\E(R)$ by [10, Theorem 3.3], so $\Index(R_{\mathfrak{p}}) \leq \Index(R)$.
\end{proof}\smallskip
\begin{lemma}[9, Lemmas 1.7 and 2.2]
Let $(R,\mathfrak{m})$ be a Cohen-Macaulay local ring with canonical module $\omega$. Let $x \in \mathfrak{m}$ be a nonzerodivisor such that $x \in \Tr_{R}\omega$, and let $M$ be an indecomposable maximal Cohen-Macaulay $R$-module such that $M \ncong R$. Then $M/xM$ has no $R/xR$-summands.
\end{lemma}\smallskip
The proof of the following lemma is contained in the proof of [9, Proposition 2.4]. We state it here for convenience.\smallskip
\begin{lemma}[9, Proposition 2.4]
Let $(R,\mathfrak{m})$ be a Cohen-Macaulay local ring with canonical module $\omega$. Let $x \in \mathfrak{m}$ be a nonzerodivisor such that $x \in \Tr_{R}\omega$ and let $\overline{R}=R/xR$. Then\newline
\[
\Index(R) \leq \Index(\overline{R}).
\]
\end{lemma}\smallskip
\begin{proof}
Let $\overline{\mathfrak{m}}=\mathfrak{m}/xR$ and $n \geq 1$. We have $\overline{R}/\overline{\mathfrak{m}}^n \cong R/(x+\mathfrak{m}^n)$. Suppose $\delta_{R}(R/\mathfrak{m}^n)=0$. Then we have a maximal Cohen-Macaulay $R$-module $N$ with no free direct summands and a surjective $R$-map \[N \longrightarrow R/\mathfrak{m}^n.\]\text{}\newline Tensoring with $\overline{R}$ gives us a surjective $\overline{R}$-map\newline
\[
N/xN \longrightarrow R/(x+\mathfrak{m}^n).
\]\text{}\newline By Lemma 2.8, $N/xN$ is a maximal Cohen-Macaulay $\overline{R}$-module with no $\overline{R}$-summands. Therefore, $\delta_{\overline{R}}(\overline{R}/\overline{\mathfrak{m}}^n)=0$. It follows that $\Index(R) \leq \Index(\overline{R})$. 
\end{proof}\smallskip
\begin{definition}[14, Definitions 8.5.1 and 8.5.10] Let $(R,\mathfrak{m})$ be a local ring. An element $x \in \mathfrak{m}$ is a {\it{superficial element}} (of order one) if there is a positive integer $c$ such that for all $n \geq c$, we have
\[
(\mathfrak{m}^{n+1}:_R x) \cap \mathfrak{m}^c=\mathfrak{m}^n.
\] A sequence $x_1,...,x_s \in \mathfrak{m}$ is a {\it{superficial sequence}} if for all $1 \leq i \leq s$, the image of $x_i$ in $\mathfrak{m}/(x_1,...,x_{i-1})$ is a superficial element of $R/(x_1,...,x_{i-1})$.
\end{definition}\smallskip
\begin{remark} If $(R,\mathfrak{m},k)$ is a local ring with infinite residue field, then $R$ has a superficial element [14, Proposition 8.5.7].  By [14, Lemma 8.5.4], every superficial element of a local ring with positive depth is a nonzerodivisor. Therefore, a $d$-dimensional Cohen-Macaulay local ring $(R,\mathfrak{m},k)$ with infinite residue field has a superficial sequence ${\bf{x}}=x_1,...,x_d \in \mathfrak{m}$ of length $d$, and this sequence is a regular sequence.
\end{remark}\smallskip
\begin{definition}[16, page 107] For an $R$-module $M$, let $\ell(M)$ denote the length of $M$. For integers $n \gg 0$, the {\it{Samuel function}} $\chi_{R}(n):=\ell(R/\mathfrak{m}^{n+1})$ is a polynomial in $n$ with rational coefficients of the form \[\chi_{R}(n)=\frac{e}{d!}n^d+(\text{terms of lower degree}),\] where $e=e(R)$ is the Hilbert-Samuel multiplicity of $R$. \end{definition}\smallskip
\begin{lemma} Let $(R,\mathfrak{m},k)$ be a $d$-dimensional Cohen-Macaulay local ring with infinite residue field. Let $x_{1},...,x_{d} \in \mathfrak{m}$ be a superficial sequence. Then for $1 \leq s \leq d$, we have \[e(R/(x_{1},...,x_{s}))=e(R).\]\end{lemma}
\begin{proof} It is enough to prove the result for $s=1$. Let $\overline{R}=R/x_{1}R$ and $\overline{\mathfrak{m}}=\mathfrak{m}/x_{1}R$. By Remark 2.11, $x_{1}$ is $R$-regular, so $\text{dim}\,\overline{R}=d-1$. By the remark following [18, page 286, Lemma 4], for $n \gg 0$, we have \[\ell(\overline{R}/\overline{\mathfrak{m}}^{n+1})=\ell(R/\mathfrak{m}^{n+1})-\ell(R/\mathfrak{m}^{n})+c\] for some constant $c$. Therefore, \newline \[\ell(\overline{R}/\overline{\mathfrak{m}}^{n+1})=\frac{e}{d!}[n^{d}-(n-1)^{d}]+(\text{polynomial of degree } d-2 \text{ in } n)\]\[=\frac{e}{(d-1)!}n^{d-1}+(\text{polynomial of degree } d-2 \text{ in } n).\] It follows that $e(\overline{R})=e(R)$.\end{proof}\smallskip
\begin{remark}
For a Gorenstein local ring $(R,\mathfrak{m},k)$ with infinite residue field, Ding proved that $R$ is an abstract hypersurface if and only if $\Index(R)=e(R)$ [10, Theorem 3.3]. Motivated by this result, we prove the following.
\end{remark}\smallskip
\begin{proposition}
Let $(R,\mathfrak{m},k)$ be a $d$-dimensional Cohen-Macaulay local ring with infinite residue field and canonical module $\omega$ such that $\mathfrak{m} \subseteq \Tr_{R}\omega$. Let ${\bf{x}}=x_1,...,x_s \in \mathfrak{m}$ be a superficial sequence, where $1 \leq s \leq d$, and let $\overline{R}=R/{\bf{x}}R$. If $\Index(R)=\E(R)$, then 
\[
\Index(R)=\Index(\overline{R})=\E(\overline{R}).
\]
\end{proposition}\smallskip
\begin{proof}
It is enough to prove the result for $s=1$. By Lemma 2.9 and Remark 2.11, we have $\Index(R) \leq \Index(\overline{R})$. By Corollary 2.4, we have $\Index(\overline{R}) \leq \E(\overline{R})$ and by Lemma 2.13, we have $\E(\overline{R})=\E(R)$. Therefore, 
\[
\E(R)=\Index(R) \leq \Index(\overline{R}) \leq \E(\overline{R})=\E(R).
\]
\end{proof}
\begin{definition}
Let $(R,\mathfrak{m})$ be a $d$-dimensional Cohen-Macaulay local ring. Let $g=\Gll(R)$. A maximal regular sequence ${\bf{x}}=x_1,...,x_d \in \mathfrak{m}$ is a {\it{witness}} to $\Gll(R)$ if $\mathfrak{m}^g \subseteq {\bf{x}}R$.
\end{definition}\smallskip
\begin{lemma}
Let $(R,\mathfrak{m})$ be a $d$-dimensional Cohen-Macaulay local ring. Let ${\bf{x}}=x_1,...,x_r \in \mathfrak{m}$ be a regular sequence that is part of a witness to $\Gll(R)$. Let $\overline{R}=R/{\bf{x}}R$. Then 
\[
\Gll(R)=\Gll(\overline{R}).
\]
\end{lemma}
\begin{proof}
It is enough to prove the result for $r=1$. Let $\overline{\mathfrak{m}}=\mathfrak{m}/x_1 R$. Let $g=\Gll(R)$ and $g'=\Gll(\overline{R})$. Suppose $x_1$ is part of witness $x_1,...,x_d$ to $\Gll(R)$. Since $\mathfrak{m}^g \subseteq (x_1,...,x_d)R$, we have $\overline{\mathfrak{m}}^g \subseteq (x_2,...,x_d)\overline{R}$. Therefore, $\Gll(\overline{R}) \leq \Gll(R)$. On the other hand, if $\overline{y}_2,...,\overline{y}_d$ is a regular sequence in $\overline{R}$ such that $\overline{\mathfrak{m}}^{g'} \subseteq (\overline{y}_2,...,\overline{y}_d)\overline{R}$, then $\mathfrak{m}^{g'} \subseteq (x_1, y_2,...,y_d)R$. So $\Gll(R) \leq \Gll(\overline{R})$.
\end{proof}\smallskip
\begin{remark}
Ding also proved that if $(R,\mathfrak{m},k)$ is a Gorenstein local ring with Cohen-Macaulay associated graded ring and infinite residue field, then $\Index(R)=\Gll(R)$ [11, Theorem 2.1]. Motivated by this result, we prove the following.
\end{remark}\smallskip
\begin{proposition}
Let $(R,\mathfrak{m})$ be a Cohen-Macaulay local ring with canonical module $\omega$ such that $\mathfrak{m} \subseteq \Tr_{R}\omega$. Let ${\bf{x}}=x_1,...,x_r \in \mathfrak{m}$ be a regular sequence that is part of a witness to $\Gll(R)$. Let $\overline{R}=R/{\bf{x}}R$. If $\Index(R)=\Gll(R)$, then \[\Index(R)=\Index(\overline{R})=\Gll(\overline{R}).\]
\end{proposition}
\begin{proof}
It is enough to prove the result for $r=1$. By Corollary 2.4, we have $\Index(\overline{R}) \leq \Gll(\overline{R})$. By Lemma 2.9, we have $\Index(R) \leq \Index(\overline{R})$ and by Lemma 2.17, we have $\Gll(R)=\Gll(\overline{R})$. It follows that
\[
\Gll(R)=\Index(R) \leq \Index(\overline{R}) \leq \Gll(\overline{R})=\Gll(R).
\]
\end{proof}
In [1], we proved the following theorem and corollary generalizing [11, Theorem 2.1] for one-dimensional Cohen-Macaulay local rings.\newline
\begin{theorem}[1, Theorem 2.3] Let $(R,\mathfrak{m})$ be a one-dimensional Cohen-Macaulay local ring for which $\Index(R)$ is finite. Let $s=\Index(R)$ and $x \in \mathfrak{m}^{t} \setminus \mathfrak{m}^{t+1}$ a nonzerodivisor, where $t \geq 1$. If the induced map
\[\overline{x}:\mathfrak{m}^{i-1}/\mathfrak{m}^{i} \longrightarrow \mathfrak{m}^{i+t-1}/\mathfrak{m}^{i+t}\]
is injective for $1 \leq i \leq s$, then \[\Gll(R) \leq \Index(R)+t-1.\]\text{}\newline If $\mathfrak{m}^{s+t-1}$ is not a principal ideal, then $\mathfrak{m}^{s+t-1} \subseteq xR.$
\end{theorem}\text{}
\begin{corollary}[1, Corollary 2.5] Let $(R,\mathfrak{m})$ be a one-dimensional Cohen-Macaulay local ring with canonical module $\omega$ such that $\mathfrak{m} \subseteq \Tr_{R}(\omega)$. Let $x \in \mathfrak{m}^{t} \setminus \mathfrak{m}^{t+1}$ such that the initial form $x^{*} \in \Gr_{\mathfrak{m}}(R)$ is a regular element. Then \[\Index(R) \leq \Gll(R) \leq \Index(R)+t-1.\]\end{corollary}\smallskip
We generalize these results to symbolic powers of prime ideals of height one in a Cohen-Macaulay local ring $R$ of positive dimension. For $n \geq 1$ and $\mathfrak{p} \in \Spec(R)$, we let $\mathfrak{p}^{(n)}$ denote the $n^{\text{th}}$ symbolic power of $\mathfrak{p}$, given by $\mathfrak{p}^{(n)}:=R \cap \mathfrak{p}^{n}R_{\mathfrak{p}}$. In Corollary 2.23 below, we let $F$ denote the filtration given by the descending chain $R=\mathfrak{p}^{(0)} \supseteq \mathfrak{p}^{(1)} \supseteq \mathfrak{p}^{(2)} \supseteq \cdot\cdot\cdot$ of symbolic primes. We denote by $\Gr_{F}(R)$ the associated graded ring of $R$ with respect to $F$. For an element $x \in \mathfrak{p}^{(t)} \setminus \mathfrak{p}^{(t+1)}$, we let $x^{*}=x+\mathfrak{p}^{(t+1)} \in \mathfrak{p}^{(t)}/\mathfrak{p}^{(t+1)}$.\newline
\begin{proposition}
Let $(R,\mathfrak{m})$ be a Cohen-Macaulay local ring with $\mathfrak{p} \in \Spec(R)$ such that $\Height(\mathfrak{p})=1$ and $\Index(R_{\mathfrak{p}})$ is finite. Let $s=\Index(R_{\mathfrak{p}})$ and $x \in \mathfrak{p}^{(t)} \setminus \mathfrak{p}^{(t+1)}$ a nonzerodivisor, where $t \geq 1$. If the induced map
\[\overline{x}:\mathfrak{p}^{(i-1)}/\mathfrak{p}^{(i)} \longrightarrow \mathfrak{p}^{(i+t-1)}/\mathfrak{p}^{(i+t)}\]
is injective for $1 \leq i \leq s$, then \[\Gll(R_{\mathfrak{p}}) \leq \Index(R_{\mathfrak{p}})+t-1.\]
\end{proposition}\smallskip
\begin{proof}
Since $\Height(\mathfrak{p})=1$, we have that $(R_{\mathfrak{p}},\mathfrak{p}R_{\mathfrak{p}})$ is a one-dimensional Cohen-Macaulay local ring. Since $x \in \mathfrak{p}$ is $R$-regular, its image in $\mathfrak{p}R_{\mathfrak{p}}$ is $R_{\mathfrak{p}}$-regular. Since the maps\smallskip 
\[
\overline{x}:\mathfrak{p}^{(i-1)}/\mathfrak{p}^{(i)} \longrightarrow \mathfrak{p}^{(i+t-1)}/\mathfrak{p}^{(i+t)}
\]\text{} are injective for $1 \leq i \leq s$, the localization of each map at $\mathfrak{p}$\newline
\[
\overline{x}:(\mathfrak{p}^{i-1}R_{\mathfrak{p}})/(\mathfrak{p}^{i}R_\mathfrak{p}) \longrightarrow (\mathfrak{p}^{i+t-1}R_{\mathfrak{p}})/(\mathfrak{p}^{i+t}R_{\mathfrak{p}})
\]\text{}\newline is also injective for $1 \leq i \leq s$. By Theorem 2.20, we have $\Gll(R_{\mathfrak{p}}) \leq \Index(R_{\mathfrak{p}})+t-1$. 
\end{proof}\smallskip
\begin{corollary}
Let $(R,\mathfrak{m})$ be a Cohen-Macaulay local ring with canonical module $\omega$. Let $\mathfrak{p} \in \Spec(R)$ such that $\Height(\mathfrak{p})=1$ and $\mathfrak{p}R_{\mathfrak{p}} \subseteq \Tr_{R}(\omega)R_{\mathfrak{p}}$. Let $x \in \mathfrak{p}^{(t)} \setminus \mathfrak{p}^{(t+1)}$ such that $x^* \in \Gr_{F}(R)$ is a regular element. Then 
\[
\Index(R_{\mathfrak{p}}) \leq \Gll(R_{\mathfrak{p}}) \leq \Index(R_{\mathfrak{p}})+t-1.
\]
\end{corollary}
\begin{proof}
This follows from Corollary 2.7 and Proposition 2.22.
\end{proof}\smallskip
\begin{remark}
Letting $(R,\mathfrak{m})$ be a one-dimensional Cohen-Macaulay local ring and $\mathfrak{p}=\mathfrak{m}$ in Corollary 2.23, we recover Corollary 2.21.
\end{remark}\text{}
\section{\large{Criteria for Elias and Burch ideals}} \text{} 
In this section, we prove criteria for when ideals in a local ring are Elias and Burch. We are motivated in part by [7, Proposition 2.10], where De Stefani effectively studies when powers of the maximal ideal in a one-dimensional Gorenstein local ring are Elias. 

We are also motivated by [5, Proposition 3.7], where Dao proves that, for a pair of ideals $I$ and $J$ in a local ring $(R,\mathfrak{m})$, the sum $I+J\mathfrak{m}$ is Burch if $\mathfrak{m}J \nsubseteq \mathfrak{m}I$. When $R$ has positive depth, we show that $I$ is Burch if $\mathfrak{m}J \nsubseteq \mathfrak{m}I$ and $IJ \subseteq xR$ for some nonzerodivisor $x \in \mathfrak{m}$. We study the relationship between Elias, Burch, and Ulrich ideals, and use our criteria to give examples of ideals in numerical semigroup rings that are Elias and Burch. 
\newline\newline
Let $(R,\mathfrak{m})$ be a one-dimensional Cohen-Macaulay local ring and $M$ an $R$-module. The {\it{type}} of $M$, denoted $\text{type}_{R}(M)$ or simply $\text{type}(M)$, is given by \smallskip\[\text{type}_{R}(M):=\text{dim}_{k}\text{Ext}_{R}^{\text{dim}_{R}M}(k,M).\]\text{}\newline For an $\mathfrak{m}$-primary ideal $I$ of $R$, we have $\text{type}_R(R/I) \leq \text{type}_R(I)$. When equality holds, we say that $I$ is {\it{Elias}}.\newline
\begin{definition}[4, Definition 1.1]
Let $(R,\mathfrak{m})$ be a one-dimensional Cohen-Macaulay local ring. Let $I$ be an $\mathfrak{m}$-primary ideal of $R$. We say that $I$ is {\it{Elias}} if $\text{type}_{R}(I)=\text{type}_{R}(R/I)$.
\end{definition}\text{}\newline
In [4], Dao relates Elias ideals to several other classes of ideals, including Ulrich ideals.\newline
\begin{definition}[4, Definition 2.1]
Let $(R,\mathfrak{m},k)$ be a one-dimensional Cohen-Macaulay local ring. Let $I$ be an $\mathfrak{m}$-primary ideal of $R$. We say that $I$ is {\it{Ulrich}} (as an $R$-module) if $\mu(I)=e(R)$.\newline\newline If $k$ is infinite, then $I$ is Ulrich if and only if $xI=\mathfrak{m}I$ for some $x \in \mathfrak{m}$ (equivalently, for any $x \in \mathfrak{m}$ such that $\ell(R/xR)=e(R)$). We call such $x$ a {\it{witness}} to $I$ as an Ulrich ideal.
\end{definition}\smallskip
\begin{remark}
Note that $xR$ is an Elias ideal for any nonzerodivisor $x \in R$. Since we have $\mathfrak{m}^s \subseteq xR$ for sufficiently large $s \geq 1$ and Elias ideals are closed under inclusion, it follows that $\mathfrak{m}^s$ is an Elias ideal for sufficiently large $s$. [4, Corollary 1.3].
\end{remark}\smallskip\smallskip
\begin{definition}[4, Definition 3.1]
The {\it{Elias index}} of $R$, denoted \text{eli}$(R)$, is the smallest positive integer $s$ such that $\mathfrak{m}^s$ is Elias. The {\it{Ulrich index}} of $R$, denoted $\Ulr(R)$, is the smallest $s$ such that $\mathfrak{m}^s$ is Ulrich. 
\end{definition}\smallskip\smallskip
We see that $\Eli(R) \leq \Gll(R)$. If the residue field $k$ is infinite, we also have $\Gll(R) \leq \Ulr(R)+1$. [4, Theorem 3.2]. If $k$ is infinite and the associated graded ring $\Gr_{\mathfrak{m}}(R)$ is Cohen-Macaulay, then both inequalities are equalities.\smallskip\smallskip
\begin{theorem}[4, Theorem 3.2]
Let $(R,\mathfrak{m},k)$ be a one-dimensional Cohen-Macaulay local ring. Suppose the associated graded ring $\Gr_{\mathfrak{m}}(R)$ is Cohen-Macaulay and $k$ is infinite. Then $\Eli(R)=\Gll(R)=\Ulr(R)+1$. 
\end{theorem}\smallskip\smallskip
We use Theorem 3.5 to answer a special case of the following question of De Stefani. Recall that, for ideals $J \subseteq I$ of a local ring $R$, we say that $J$ is a {\it{reduction}} of $I$ if $I^{n+1}=JI^n$ for some non-negative integer $n$.\smallskip\smallskip
\begin{question}[7, Question 4.5 (ii)] For a Gorenstein local ring $R$ with infinite residue field, is $\Gll(R)$ always attained by a system of parameters that generates a reduction of $\mathfrak{m}$?
\end{question}\smallskip
\begin{definition}[14, Definition 8.2.3]
Let $(R,\mathfrak{m})$ be a local ring. Let $J \subseteq I$ be ideals of $R$ with $J$ a reduction of $I$. The {\it{reduction number}} of $I$ with respect to $J$ is the minimum integer $n$ such that $JI^n=I^{n+1}$.
\end{definition}\smallskip\smallskip
\begin{corollary}
Let $(R,\mathfrak{m},k)$ be a one-dimensional Cohen-Macaulay local ring with infinite residue field and $\Gr_{\mathfrak{m}}(R)$ Cohen-Macaulay. Then we have the following.\newline
\begin{enumerate}
\item $\Gll(R)$ is attained by an element $x \in \mathfrak{m}$ that generates a reduction of $\mathfrak{m}$ with reduction number $\Ulr(R)$.
\item The set of witnesses $x$ to $\Gll(R)$ in the vector space $\mathfrak{m}/\mathfrak{m}^2$ contains a non-empty Zariski open set.
\end{enumerate}
\end{corollary}
\begin{proof}
By Theorem 3.5, we have $\Gll(R)=\Ulr(R)+1$. Let $g=\Gll(R)$. Then $\mathfrak{m}^{g-1}$ is an Ulrich ideal, so by Definition 3.2, there is an element $x \in \mathfrak{m}$ such that $x\mathfrak{m}^{g-1}=\mathfrak{m}^{g}$. In particular, $\mathfrak{m}^{g} \subseteq xR$, so $x$ is a witness to $\Gll(R)$ that generates a reduction of $\mathfrak{m}$. Let $r$ denote the reduction number of $\mathfrak{m}$ with respect to $x$. We see that $r \leq g-1$. Suppose $r<g-1$. We have $\mathfrak{m}^{r+1}=x\mathfrak{m}^{r} \subseteq xR$. Since $r+1<g$, this contradicts the fact that $g=\Gll(R)$. Therefore, $r=g-1=\Ulr(R)$. By [4, Remark 2.2], the set of witnesses $x$ to $\mathfrak{m}^{g-1}$ as Ulrich is a non-empty Zariski open set in $\mathfrak{m}/\mathfrak{m}^2$. Since each of these elements is also a witness to $\Gll(R)$, we see that the set of witnesses to $\Gll(R)$ in $\mathfrak{m}/\mathfrak{m}^2$ contains a non-empty Zariski open set. 
\end{proof}\smallskip
\begin{question}
Let $(R,\mathfrak{m},k)$ be a one-dimensional Cohen-Macaulay local ring with infinite residue field and $\Gr_{\mathfrak{m}}(R)$ Cohen-Macaulay. Is the set of witnesses to $\Gll(R)$ in the vector space $\mathfrak{m}/\mathfrak{m}^2$ a Zariski open set?
\end{question}\smallskip
In what follows, we let $Q$ denote the total ring of fractions of a ring $R$.\smallskip
\begin{remark} Suppose $(R,\mathfrak{m})$ is a one-dimensional Gorenstein local ring. Let $x \in \mathfrak{m}$ be a nonzerodivisor and $I$ an $\mathfrak{m}$-primary ideal of $R$. Then $xR$ is a canonical ideal of $R$, so by [4, Theorem 1.2], $I$ is Elias if and only if $x \in \mathfrak{m}(xR:_Q I)$. If $(R,\mathfrak{m})$ is a one-dimensional Cohen-Macaulay local ring, then we still have that an $\mathfrak{m}$-primary ideal $I$ is Elias if $x \in \mathfrak{m}(xR:_Q I)$ for some nonzerodivisor $x \in \mathfrak{m}$.\end{remark}\smallskip
\begin{proposition}[c.f. 7, Proposition 2.10]
Let $(R,\mathfrak{m})$ be a one-dimensional Cohen-Macaulay local ring and $I$ an $\mathfrak{m}$-primary ideal of $R$. If $x \in \mathfrak{m}(xR:_Q I)$ for some nonzerodivisor $x \in \mathfrak{m}$, then $I$ is Elias.
\end{proposition}
\begin{proof}
Suppose $x \in \mathfrak{m}(xR:_Q I)$. Then
\[
(I:_Q \mathfrak{m}) \subseteq ((xR:_Q (xR:_Q I)):_Q \mathfrak{m})=(xR:_Q (xR:_Q I)\mathfrak{m}) 
\] and 
\[
(xR:_Q (xR:_Q I)\mathfrak{m}) \subseteq (xR:_Q xR) = R.
\] Therefore $(I:_Q \mathfrak{m}) \subseteq R$ and $I$ is Elias by [4, Theorem 1.2].
\end{proof}\smallskip
\begin{remark}
If $R$ is not Gorenstein and $I$ is an Elias ideal of $R$, there may be nonzerodivisors $x \in I$ such that $x \not\in \mathfrak{m}(xR:_Q I)$.
\end{remark}\smallskip
\begin{example}
Let $k$ be a field and let \[R=k\llbracket t^4,t^5, t^{11}\rrbracket \cong k\llbracket a,b,c \rrbracket/(a^4-bc, b^3-ac, c^2-a^3b^2)\] with maximal ideal $\mathfrak{m}=(t^4,t^5,t^{11})R$. In [4, Example 2.8], Dao showed that $\mathfrak{m}^2$ is an Elias ideal of $R$. Consider $(t^{9}R:_Q \mathfrak{m}^2)$. Since $t^9 \in \mathfrak{m}^2$, we have $(t^{9}R:_Q \mathfrak{m}^2)=(t^{9}R:_R \mathfrak{m}^2)$. It is clear that $(t^{9}R:_R \mathfrak{m}^2) \subseteq \mathfrak{m}$. We claim that $(t^{9}R:_R \mathfrak{m}^2) \subseteq \mathfrak{m}^2$. Indeed, we have $t^4 \not\in (t^{9}R:_R \mathfrak{m}^2)$, since $t^4t^8=t^{12} \not\in t^{9}R$. We have $t^5 \not\in (t^{9}R:_R \mathfrak{m}^2)$, since $t^5t^{10}=t^{15} \not\in t^{9}R$. Finally, $t^{11} \not\in (t^{9}R:_R \mathfrak{m}^2)$, since $t^{11}t^8=t^{19} \not\in t^9R$. Therefore, $(t^{9}R:_R \mathfrak{m}^2) \subseteq \mathfrak{m}^2$ and $\mathfrak{m}(t^{9}R:_R \mathfrak{m}^2) \subseteq \mathfrak{m}^3$, so $t^9 \not\in \mathfrak{m}(t^{9}R:_R \mathfrak{m}^2)$. Now consider the ideal $I=(t^8, t^9, t^{15}, t^{16}, t^{22})R$. One can check that $t^5 \in (t^9:_Q I)$, so $t^9 \in \mathfrak{m}(t^9:_Q I)$. Therefore, $I$ is Elias by Proposition 3.11. Of course, $I$ is also Elias because $I \subseteq \mathfrak{m}^2$.
\end{example}\text{}\newline For nonzerodivisors $x \in \mathfrak{m} \setminus \mathfrak{m}^2$, we have the following characterization of ideals $I$ containing $x$ for which $x \in \mathfrak{m}(xR:_Q I)$. The proof follows the argument in the proof of [4, Corollary 1.3 (5)].\text{}\newline
\begin{proposition}[c.f. 4, Corollaries 1.3 and 4.3]
Let $(R,\mathfrak{m})$ be a one-dimensional Cohen-Macaulay local ring. Let $I$ be an $\mathfrak{m}$-primary ideal of $R$ and $x \in \mathfrak{m} \setminus \mathfrak{m}^2$ a nonzerodivisor such that $x \in I$. Then $x \in \mathfrak{m}(xR:_Q I)$ if and only if $I=xR$. 
\end{proposition}
\begin{proof}
Suppose $xR \subsetneq I$. Since $x \in I$, we have $(xR:_Q I)=(xR:_R I)$. Suppose we have $(xR:_R I) \not\subseteq \mathfrak{m}$. Let $r \in (xR:_R I) \setminus \mathfrak{m}$. Then $rI \subseteq xR$, so $I \subseteq xR$ and $I=xR$, a contradiction. It follows that $(xR:_Q I) \subseteq \mathfrak{m}$ and $\mathfrak{m}(xR:_Q I) \subseteq \mathfrak{m}^2$. Therefore, $x \not\in \mathfrak{m}(xR:_Q I)$. Conversely, if $I=xR$, then $(xR:_Q I)=R$ and $x \in \mathfrak{m}(xR:_Q I)$. 
\end{proof}\text{}
\begin{corollary}
Let $(R,\mathfrak{m})$ be a one-dimensional Cohen-Macaulay local ring. Let $I$ and $J$ be ideals of $R$ such that $I$ is $\mathfrak{m}$-primary and $IJ \subseteq xR$ for some nonzerodivisor $x \in \mathfrak{m}$. If $x \in J\mathfrak{m}$, then $I$ is Elias.
\end{corollary}
\begin{proof}
Suppose $x \in J\mathfrak{m}$. Since $J \subseteq (xR:_Q I)$, we have $x \in \mathfrak{m}(xR:_Q I)$. Therefore, $I$ is Elias by Proposition 3.11.
\end{proof}\text{}
\begin{definition}[6, Definition 2.1]
Let $(R,\mathfrak{m})$ be a local ring. An ideal $I$ of $R$ is {\it{Burch}} if $\mathfrak{m}I \neq \mathfrak{m}(I:_R \mathfrak{m})$.
\end{definition}\text{}
\begin{remark}
Burch ideals are abundant-namely, $I\mathfrak{m}$ is Burch for any nonzero ideal $I$ of a local ring $(R,\mathfrak{m})$. If $R$ has positive depth, then $\mathfrak{m}^t$ is Burch for each $t \geq 1$, and any integrally closed ideal of $R$ is Burch. [6, Example 2.2].
\end{remark}\text{}
\begin{proposition}
Let $(R,\mathfrak{m}, k)$ be a one-dimensional Cohen-Macaulay local ring with infinite residue field. Let $I \neq 0$ be Ulrich and Elias. Then $I$ is Burch.
\end{proposition}
\begin{proof}
Suppose $I$ is not Burch. Since $I$ is Ulrich and $k$ is infinite, we have $\mathfrak{m}I=xI$ for some nonzerodivisor $x \in \mathfrak{m}$. Since $I$ is not Burch, by [6, Proposition 2.3] we have 
\[
(I:_R \mathfrak{m})=(\mathfrak{m}I:_R \mathfrak{m})=(xI:_R \mathfrak{m}).
\] Since $I$ is Elias, we have $(xI:_R \mathfrak{m})=x(I:_R \mathfrak{m})$ by [4, Theorem 1.2]. Therefore, we have $(I:_R \mathfrak{m})=x(I:_R \mathfrak{m})$. So $(I:_R \mathfrak{m})=0$ and $I=0$, a contradiction.
\end{proof}\text{}
\begin{proposition}
Let $(R,\mathfrak{m},k)$ be a one-dimensional Cohen-Macaulay local ring with infinite residue field. Let $I$ be an Ulrich ideal of $R$. If $I^2 \neq I(I:_R \mathfrak{m})$, then $I$ is Burch.
\end{proposition}
\begin{proof}
Suppose $I$ is not Burch. Then $\mathfrak{m}I=\mathfrak{m}(I:_R \mathfrak{m})$. Since $I$ is Ulrich and $k$ is infinite, we have $Ix=I\mathfrak{m}$ for some nonzerodivisor $x \in \mathfrak{m}$. Therefore, 
\[
I^{2}x=I^2\mathfrak{m}=I\mathfrak{m}(I:_R\mathfrak{m})=xI(I:_R \mathfrak{m}).
\] It follows that $I^2=I(I:_R\mathfrak{m})$.
\end{proof}\text{}
\begin{example}
Let $k$ be a field and $R=k\llbracket t^4, t^6, t^7 \rrbracket \cong k\llbracket x,y,z \rrbracket/(x^3-y^2, z^2-x^2y)$ with maximal ideal $\mathfrak{m}=(t^4,t^6,t^7)R$. Let $I=(t^4,t^6)R$. Then $t^{4}I=\mathfrak{m}I=(t^8,t^{10})R$, so $I$ is Ulrich. Since $t^{7}\mathfrak{m}=(t^{11},t^{13},t^{14})R \subseteq I$, we have $t^{7} \in (I:_R \mathfrak{m})$ and $t^{11} \in I(I:_R \mathfrak{m})$. However, $t^{11} \not\in I^2=(t^8,t^{10})R$. Therefore, $I$ is Burch by Proposition 3.19.
\end{example}\text{}\newline
For two ideals $I$ and $J$, Dao proved the following criterion for when $I+J\mathfrak{m}$ is Burch.\newline
\begin{proposition}[5, Proposition 3.3] Let $(R,\mathfrak{m})$ be a local ring with ideals $I$ and $J$ such that $J\mathfrak{m} \not\subseteq I\mathfrak{m}$. Then $I+J\mathfrak{m}$ is Burch.
\end{proposition}\text{}\newline
In the following, we write $(I:J)$ to denote $(I:_R J)$ for ideals $I$ and $J$ of $R$.\newline
\begin{proposition}
Let $(R,\mathfrak{m})$ be a local ring with positive depth. Let $I$ and $J$ be ideals of $R$ such that $IJ\subseteq xR$ for some nonzerodivisor $x \in \mathfrak{m}$. If $J\mathfrak{m} \nsubseteq I\mathfrak{m}$, then $I$ is Burch.
\end{proposition}
\begin{proof}
Since $J \subseteq (x:I)$, we have the following.
\[
(J:\mathfrak{m})\subseteq ((x:I):\mathfrak{m})=(x:I\mathfrak{m}).
\] Suppose $I$ is not Burch. Then $I\mathfrak{m}=\mathfrak{m}(I:\mathfrak{m})$, and
\[
(x:I\mathfrak{m})=(x:\mathfrak{m}(I:\mathfrak{m}))=(x(I:\mathfrak{m}):\mathfrak{m}).
\] Since $x(I:\mathfrak{m}) \subseteq (xI:\mathfrak{m})$, we then have 
\[
(x(I:\mathfrak{m}):\mathfrak{m}) \subseteq ((xI:\mathfrak{m}):\mathfrak{m})=(xI:\mathfrak{m}^2) \subseteq (xI:x\mathfrak{m})=(I:\mathfrak{m}).
\] It follows that $(J:\mathfrak{m}) \subseteq (I:\mathfrak{m})$. Therefore, 
\[
\mathfrak{m}J \subseteq \mathfrak{m}(J:\mathfrak{m}) \subseteq \mathfrak{m}(I:\mathfrak{m})=\mathfrak{m}I.
\]
\end{proof}
\begin{example}
Let $k$ be a field and $R=k\llbracket t^4,t^6,t^7 \rrbracket$ with $\mathfrak{m}=(t^4,t^6,t^7)R$. Let $I=(t^7,t^8)R$ and $J=t^{4}R$. Then $IJ \subseteq t^{4}R$. Since $J\mathfrak{m}=(t^8,t^{10},t^{11})R$ and $I\mathfrak{m}=(t^{11},t^{12},t^{13},t^{14})R$, we have $J\mathfrak{m} \not\subseteq I\mathfrak{m}$. By Proposition 3.22, the ideal $I$ is Burch.
\end{example}\text{}
\begin{corollary}
Let $(R,\mathfrak{m})$ be a one-dimensional Cohen-Macaulay local ring. Let $I$ and $J$ be $\mathfrak{m}$-primary ideals of $R$ such that $IJ \subseteq xR$ for some nonzerodivisor $x \in \mathfrak{m}$. If $I$ is not Burch and $x \in J\mathfrak{m}$, then $I$ is Elias and $J$ is Elias.
\end{corollary}
\begin{proof}
Since $I$ is not Burch, we have $J\mathfrak{m} \subseteq I\mathfrak{m}$ by Proposition 3.22. Since $x \in J\mathfrak{m}$, we have $x \in I\mathfrak{m}$ as well. It follows from Corollary 3.15 that $I$ and $J$ are both Elias.
\end{proof}\text{}\newline
By strengthening the hypotheses in Corollary 3.24, we obtain the following criterion for when an ideal is not Burch.\newline
\begin{proposition}
Let $(R,\mathfrak{m})$ be a local ring with positive depth. Let $I$ and $J$ be ideals of $R$ such that $IJ \subseteq x^{2}R$ for some nonzerodivisor $x \in \mathfrak{m}$. If $x \in I \cap J\mathfrak{m}$, then $I$ is not Burch. 
\end{proposition}
\begin{proof}
Since $J \subseteq (x^2:I)$, we have the following.
\[
I \subseteq (I:\mathfrak{m}) \subseteq ((x^2:(x^2: I)):\mathfrak{m}) \subseteq ((x^2: J): \mathfrak{m})
\] and \[((x^2: J): \mathfrak{m})=(x^2: J\mathfrak{m}) \subseteq (x^2: x)=xR \subseteq I.
\] Therefore, $(I: \mathfrak{m})=I$ and $\mathfrak{m}I=\mathfrak{m}(I:\mathfrak{m})$, so $I$ is not Burch. 
\end{proof}\smallskip\smallskip
\begin{corollary}
Let $(R,\mathfrak{m},k)$ be a one-dimensional Cohen-Macaulay local ring. Let $I$ and $J$ be $\mathfrak{m}$-primary ideals and $x \in \mathfrak{m}$ a nonzerodivisor such that $IJ \subseteq x^{2}R$.
\begin{enumerate}
\item If $x \in I \cap J\mathfrak{m}$, then $I$ and $J$ are Elias. If also $k$ is infinite, then $I$ is not Ulrich. 
\item If $x \in J \cap I\mathfrak{m}$, then $I$ and $J$ are Elias. If also $k$ is infinite, then $J$ is not Ulrich.
\end{enumerate}
\end{corollary}
\begin{proof}
We prove (1). By Proposition 3.25, we have $I$ is not Burch. Therefore, $I$ and $J$ are Elias by Corollary 3.24. If $k$ is infinite, then $I$ is not Ulrich by Proposition 3.18.
\end{proof}\smallskip\smallskip
\begin{proposition}
Let $(R,\mathfrak{m})$ be a local ring. Let $I$, $J$, and $K$ be ideals of $R$ such that $I=(K: J)$, $J=(K: I)$, and $I\mathfrak{m}=J\mathfrak{m}$. Then $I$ is Burch if and only if $J$ is Burch.
\end{proposition}
\begin{proof}
We have the following.
\[
(J:\mathfrak{m})=((K:I):\mathfrak{m})=(K:I\mathfrak{m})=(K:J\mathfrak{m})
\] and 
\[
(K:J\mathfrak{m})=((K:J):\mathfrak{m})=(I:\mathfrak{m}).
\] It follows that $\mathfrak{m}(I:\mathfrak{m})=\mathfrak{m}(J:\mathfrak{m})$. Suppose $I$ is not Burch. Then we have 
\[
\mathfrak{m}J \subseteq \mathfrak{m}(J:\mathfrak{m})=\mathfrak{m}(I:\mathfrak{m})=\mathfrak{m}I=\mathfrak{m}J,
\] so $\mathfrak{m}J=\mathfrak{m}(J:\mathfrak{m})$, and $J$ is not Burch. Conversely, if $J$ is not Burch, we have 
\[
\mathfrak{m}J=\mathfrak{m}I \subseteq \mathfrak{m}(I:\mathfrak{m})=\mathfrak{m}(J:\mathfrak{m})=\mathfrak{m}J,
\] so $\mathfrak{m}I=\mathfrak{m}(I:\mathfrak{m})$ and $I$ is not Burch.
\end{proof}

\section*{Acknowledgements} I would like to thank Michael DeBellevue and Souvik Dey for many helpful conversations and feedback on results in this paper. I would also like to thank Michael DeBellevue for his question which led to Corollary 3.8 (2). Finally, I would like to thank Graham Leuschke for his helpful suggestions for this paper, including Question 3.9.

\bibliographystyle{amsplain}
}
\end{document}